\documentclass[11pt, letterpaper,english]{amsart}
\usepackage[left=1in,right=1in,bottom=1in,top=1in]{geometry}

\usepackage{amsmath,amsthm,amssymb}
\usepackage{mathrsfs}
\usepackage{cite}
\usepackage{hyperref}
\usepackage{enumerate}
\usepackage[usenames,dvipsnames]{color}
\usepackage{colonequals} 

\newtheorem{thm}{Theorem}[section]
\newtheorem{crl}[thm]{Corollary}

\newtheorem{lmm}[thm]{Lemma}

\newtheorem{prp}[thm]{Proposition}

\theoremstyle{definition}
\newtheorem{dfn}[thm]{Definition}
\newtheorem{exa}[thm]{Example}%%%
\newtheorem{notation}[thm]{Notation}
\theoremstyle{remark}
\newtheorem{rem}[thm]{Remark}%%% * means no numbering
\usepackage{mathrsfs}
\usepackage{url}
\newcommand{\mathfont}{\mathbf}
\newcommand{\ZZ}{\mathfont Z}
\newcommand{\QQ}{\mathfont Q}
\newcommand{\FF}{\mathfont F}
\newcommand{\SSS}{\mathfont S}

\title[Supersingular curves via the Shimura--Taniyama method]{Supersingular curves via the Shimura--Taniyama method}
\author{Jeremy Booher}
\address{University of Florida}
\email{jeremybooher@ufl.edu}
\author{Rachel Pries}
\address{Colorado State University}
\email{pries@colostate.edu}
\date{\today}

\thanks{Pries was supported by NSF grant DMS-22-00418.
This work was supported by the Research Institute for Mathematical Sciences, an International Joint Usage/Research Center located in Kyoto University
We thank the organizers of the RIMS conference on 
Theory and Applications of Supersingular Curves and Supersingular Abelian Varieties, II.
We thank Everett Howe, Wanlin Li, Elena Mantovan, and Darren Schmidt for helpful 
conversations, and the referee for helpful comments.}  %optional
\keywords{\textit{curve, Jacobian, abelian variety, positive characteristic, supersingular, Frobenius, Newton polygon, automorphism}}  

\begin{document}

\maketitle

\begin{abstract}      %optional
For a curve $X$ which is an abelian cover of the projective line branched at three points, 
we study when its reduction to positive characteristic is supersingular.  
Using the method of Shimura and Taniyama, we give a complete classification of primes $p$
for which the reduction of $X$ modulo $p$ is supersingular 
when the genus of $X$ is at most $10$.
The natural density of this set of primes is larger than expected.

2020 MSC primary: 11G20, 11M38, 14G10, 14H10, 14H40. 
Secondary 11G10, 14G15, 14H37, 14H30.
%KEEP
%11G20 Curves over finite fields
%11M38 NT: Zeta and $L$-functions, analytic theory: Zeta $L$-functions characteristic $p$
%14G10 AG: Arithmetic questions: Zeta-functions and related questions 
%14H10 Curves in algebraic geometry, Families, moduli of curves (algebraic)
%14H40 Jacobians, Pryms

%11G10 Arithmetic Geometry, Abelian varieties of dim >1 
%14G15 Algebraic geometry over finite ground fields
%14H37  	Automorphisms of curves
%14H30  	Coverings of curves, fundamental group 
\end{abstract}

\section{Introduction}

Let $k$ be an algebraically closed field of prime characteristic $p$.
An elliptic curve over $k$ is supersingular if it has no points of order $p$.
A smooth projective curve $X$ of genus $g$ over $k$ is \emph{supersingular} if its Jacobian ${\mathfont{Jac}}(X)$ is isogenous to 
a product of $g$ supersingular elliptic curves.

One good method for constructing supersingular curves is to study families of curves that are abelian covers of the projective line
$\mathfont{P}^1$.    
The degree and ramification indices of the cover determine the genus $g$.
When the curve $X$ is branched at exactly three points, it is a quotient of a Fermat curve 
and ${\mathfont{Jac}}(X)$ has complex multiplication.  
This makes it easier to study the action of Frobenius on the cohomology \cite{Yui}
and hence determine whether the reduction of $X$ modulo $p$ is supersingular.  For example, Re studied supersingular quotients of the Fermat curve with the aim of producing supersingular curves of many genera in fixed characteristic \cite{re1,re2}.

Recently, we were curious about an analogous question: given $g$, for which primes $p$ do these constructions yield a  supersingular curve of genus $g$ in characteristic $p$?
Our motivation was to show that a new construction of supersingular curves \cite{BPnewss5} yields supersingular curves of genus $5$ in characteristic $p$ for infinitely many new primes $p$.  We wrote this paper because we 
could not find a reference that states what can be determined using this method.

Let $\SSS_g$ be the set of primes $p$ for which there exists a supersingular curve $X$ of genus $g$ defined over $k=\overline{\FF}_p$
such that $X$ is an abelian cover of $\mathfont{P}^1$ branched at exactly three points. 
For $5 \leq g \leq 10$, we determine the primes for which there exists a supersingular curve of this type 
in Sections~\ref{Section4}, \ref{Sdatacyclic} and \ref{Sdatanoncyclic}.
This yields the main result of this paper, Theorem~\ref{Tmainthm}, 
which states exactly which primes are in $\SSS_g$, for $5 \leq g \leq 10$, in terms of congruence conditions.
From these congruence conditions, we determine the natural density of the set $\SSS_g$, which we state as the following corollary.
  
\begin{crl} \label{Cintro} For $5 \leq g \leq 10$, the natural density $\delta_g$ of the set $\SSS_g$ is: 
\[\begin{array} {|c|c|c|c|c|c|c|}
\hline
g &  5 &   6 &   7 &   8 &  9 &  10 \\ \hline
\ \delta_g \ & \ 25/32 \ & \ 507/512 \ & \ 3/4 \ & \ 1023/1024 \ & \ 15/16 \ & \ 31/32 \  \\ \hline
%\delta_g \approx & .781 & .990 & .750 & .999 & .938 & .969 \\ \hline
\end{array}\]
\end{crl}

%[0.781250000000000, 0.990234375000000, 0.750000000000000, 0.999023437500000, 0.937500000000000, 0.968750000000000]

The cases $g=9,10$ are especially interesting, as noted in \cite[Expectation~8.5.4]{oort05}.
The codimension of the supersingular locus in 
the moduli space $\mathcal{A}_g$
of principally polarized abelian varieties of dimension $g$ is 
$g(g+1)/2 - \lfloor g^2/4 \rfloor$.  
For $g \geq 9$, this exceeds the dimension of the moduli space $\mathcal{M}_g$ 
of smooth curves of genus $g$.
The existence of supersingular curves of genus $g \geq 9$ implies that the intersection of the 
Torelli locus and supersingular locus in $\mathcal{A}_g$ is not dimensionally transversal.
This was already known to happen; e.g., there exist supersingular curves of every genus $g$ when $p=2$ \cite{vdgvdv}.
Corollary~\ref{Cintro} shows that this dimensionally non-transversal intersection occurs for most primes when $g=9$ and $g=10$, 
with $\delta_9 \approx .938$ and $\delta_{10} \approx .969$.

Work of Yui \cite{Yui} implies that $\delta_g$ is 
positive for all $g \geq 1$, which can also be deduced from our analysis, see Remark~\ref{Rpositivedensity}.  

For the proof, we use Ekedahl's work \cite[p.\ 173]{Ekedahl87} to classify the 
degrees and ramification indices of abelian covers $X \to \mathfont{P}^1$ branched at $3$ points, such that $X$ has genus $g$. 
When $g=4,5$, Ekedahl determines conditions on $p$ for which $X$ is superspecial 
(meaning ${\mathfont{Jac}}(X)$ is isomorphic to a product of supersingular elliptic curves).
We determine the set of primes $p$ for which the reduction is supersingular 
using the Shimura--Taniyama formula \cite[Section 5]{Tate71}, a method also used in \cite{LMPT1}.  
As explained in Section~\ref{Scyclic2}, 
this method exploits the structure of complex multiplication on ${\mathfont{Jac}}(X)$ to determine its Newton polygon.

We use SageMath to enumerate all abelian covers of $\mathfont{P}^1$ ramified at three points with genus $g$ satisfying $5 \leq g \leq 10$.  We restrict to the case $g \geq 5$ and $p$ odd because of previous constructions of supersingular curves:
for $g=1$ by Deuring \cite{deuring}; for $g=2$ by Serre \cite[Th\'eor\`eme 3]{serre82}, also \cite[Proposition~3.1]{IKO};
for $g=3$ by Oort \cite[Theorem~5.12]{oorthypsup}, also \cite[Theorem~1]{ibukiyama93};
for $g=4$ by Harashita, Kudo, and Senda \cite[Corollary~1.2]{khs20}, also \cite[Theorem~1.1]{P:somecasesOort}; and 
for $p=2$ and arbitrary $g$, by van der Geer and van der Vlugt \cite{vdgvdv}. 

\begin{rem}
Many mathematicians studied supersingular curves of genus $g\geq 5$ in characteristic $p$ for various pairs $(g,p)$.
Here is an incomplete list: Hermitian curves \cite{MR1186416}, Artin-Schreier curves \cite{MR1189892, vdgvdv, bouwWIN}, 
cyclic covers of $\mathfont{P}^1$ \cite{re1,re2,LMPT1, LMPT2};
and trigonal curves \cite{kudo19, kh22}.
\end{rem}

\section{Abelian covers branched at three points}

Let $X$ be a smooth projective connected curve over $k$.
Let $\phi : X \to {\mathfont{P}}^1$ be a Galois cover branched exactly at $B:=\{0,1, \infty\}$.  Let $g$ be the genus of $X$.
We suppose that the Galois group $G$ of $\phi$ is abelian and that $\phi$ is tamely ramified.

\subsection{Ramification and Inertia Types} \label{ss:types}

For $b \in B$, let $I_b$ be the inertia group above $b$ and let $c_b = |I_b|$. 
Since $\phi$ is tamely ramified, $I_b$ is cyclic.  
Without loss of generality, we suppose that $c_0 \geq c_1 \geq c_\infty$.

\begin{dfn}
The \emph{ramification type} for the cover $\phi$ is the tuple $z=[c_0,c_1,c_\infty,s]$, where $s$ is the index of $I_0$ in $G$.
\end{dfn}

The ramification type $z$ is essentially what Ekedahl studied, 
except that we switch the role of $0$ and $\infty$ to simplify the equation for $\phi$. 

\begin{notation}
Let $r=\mathrm{gcd}(c_0,c_1,c_\infty)$.  
For distinct $i,j \in B$, let $e_{i,j} =  \mathrm{gcd}(c_i,c_j)/r$.
\end{notation}

Ekedahl \cite[page 173]{Ekedahl87} verified that $e_{0,1}$, $e_{0,\infty}$, and $e_{1, \infty}$ are pairwise relatively prime, 
that $e_{0,1} \geq e_{0, \infty} \geq e_{1, \infty}$; and that
\begin{equation} \label{eq:formulaeij}
c_0 = re_{0,1}e_{0, \infty}, \quad c_1=re_{0,1}e_{1, \infty}, \quad c_\infty = re_{0,\infty}e_{1, \infty}.
\end{equation}

\begin{lmm} \cite[page 173]{Ekedahl87}
With notation as above:
\begin{equation} \label{Eformulag}
2g-2 = s(re_{0,1} e_{0,\infty} - (1 + (e_{0,1}/e_{1,\infty}) + (e_{0, \infty}/e_{1,\infty})).
\end{equation}
\end{lmm}

\begin{proof}
Above $b \in B$, the number of ramification points is $\#G/c_b$ and each has inertia group of order $c_b$.
By definition, $|G|=sc_0$. 
By the Riemann--Hurwitz formula:
\begin{eqnarray*}
2g-2 & = & -2 |G| + \sum_{b \in B} (|G|/c_b)(c_b-1) \\
& = &  |G| (1 - ((1/c_0) + (1/c_1) + (1/c_\infty))) \\
& = & s(c_0 - (1 + (c_0/c_1) + (c_0/c_\infty))).
\end{eqnarray*}
The result follows from \eqref{eq:formulaeij}.
\end{proof}

A cover $\phi$ can be described via the numerical data of the ramification type.  
However, there are also group-theoretic constraints on the ramification.
To state these, we identify $G$ with $\mathrm{Gal}(\phi)$ and define the inertia type.  

\begin{dfn} \label{dfn:inertialtype}
Following Bouw \cite[Definitions~3.2 \& 4.1]{Bouw}, the \emph{inertia type} of $\phi$ is the triple $a=(a_0,a_1,a_\infty)$, where $a_b$ denotes the
\emph{canonical generator} for $I_b$, found from the action of inertia on a local parameter at a ramification point above $b$. 
\end{dfn}

There are three constraints on the inertia type: $|a_b| = c_b$ for $b \in B$; 
$a_0, a_1, a_\infty$ generate $G$; and $a_0a_1a_\infty = \mathrm{Id}_G$.
This implies that $s$ divides $c_1$ and $s$ divides $c_\infty$.
Given a tuple $z=[c_0,c_1,c_\infty,s]$ satisfying the numerical constraints for a ramification type, 
it is apriori possible that there is no tuple $a$ satisfying the group-theoretic constraints for the inertia type. 
See Remark~\ref{Rnovalid} for some examples.

Given an inertia type, we can change the labeling of branch points whose inertia groups have the same order, 
permuting those entries of the inertia type.
Also, we can change the identification of $G$ with $\mathrm{Gal}(\phi)$;
this modifies each entry of the inertia type by some $\sigma \in \mathrm{Aut}(G)$.
Two inertia types $a$ and $a'$ are \emph{equivalent}, denoted $a \approx a'$, if they are in the same orbit under these actions.

\subsection{Cyclic covers branched at three points} \label{Scyclic1}

Suppose $\phi : X \to {\mathfont{P}}^1$ is cyclic of order $m$, and identify $G = \ZZ/m\ZZ$.  
The inertia type $(a_0,a_1,a_\infty)$ is defined up to multiplication by $(\ZZ/m\ZZ)^*$.  
Note that $c_b = |a_b| = m/\gcd(m,a_b)$.  An affine equation for the cover is 
\begin{equation} \label{eq:cycliccover}
X: y^m=x^{a_0}(1-x)^{a_1},
\end{equation}
with $G$ acting by multiplication by $m$th roots of unity on $y$.  

\begin{dfn} \label{Dsignature}
For $1 \leq j < m$, let $f_j$ be the dimension of the $\zeta_m^j$-eigenspace of $H^0(X, \Omega^1_X)$. The \emph{signature type} of $\phi$ is $f=(f_1, \ldots, f_{m-1})$.
\end{dfn}

Let $\langle q\rangle$ denote the fractional part of $q \in \QQ$.
Then $f_j= -1+\sum_{b\in B}\langle\frac{-ja_b}{m}\rangle$ by e.g., \cite[Lemma~2.7, Section~3.2]{moonen}.  
For degree $m$ cyclic covers $\phi: X \to {\mathfont{P}}^1$ branched at $B=\{0,1,\infty\}$, 
the following facts are well-known (e.g., \cite[Lemma 3.1]{LMPT1}):
\begin{itemize}
\item  
the Jacobian ${\mathfont{Jac}}(X)$ has complex multiplication by $\prod_{d} \QQ(\zeta_d)$, 
where the product is taken over all integers $d$ such that 
$1<d\mid m$ and $d\nmid a_b$ for any $b\in B$;
\item if $1 \leq j < m$, then
 $f_j \in \{0,1\}$,
and $f(j) + f(m-j) = 1$ provided $d_j$, the order of $j$ modulo $m$, does not divide $a_b$ for any $b \in B$.
\end{itemize}

\subsection{Newton polygons} 

If $X$ is a smooth curve over $k$, 
we describe the Newton polygon of its Jacobian ${\mathfont{Jac}}(X)$ in terms of 
the decomposition of the $p$-divisible group ${\mathfont{Jac}}(X)[p^\infty]$.

\begin{dfn}
For non-negative integers $c$ and $d$ with ${\mathrm{gcd}}(c,d)=1$, 
fix a $p$-divisible group $G_{c,d}$ of codimension $c$, dimension $d$, and thus height $c + d$. 
The slope of $G_{c,d}$ is $\lambda = d/(c + d)$, with multiplicity $c+d$.  Let $G_\lambda = G_{c,d}$.
\end{dfn}

Suppose $A$ is a principally polarized abelian variety of dimension $g$ over $k$.
By the Dieudonn\'e--Manin classification \cite{man63}, the 
$p$-divisible group $A[p^\infty]$ is isogenous to $\oplus_\lambda G_\lambda^{m_\lambda}$ for some 
multi-set of non-negative integers $m_\lambda$. 
This multi-set defines a symmetric Newton polygon of height $2g$ with integer break-points.
As a short-hand, we write $\nu(A) = \oplus_\lambda G_\lambda^{m_\lambda}$ 
and call both $\nu(A)$ and its visual representation the \emph{Newton polygon} of $A$.

The Newton polygon of a curve $X$ over $k$ is the Newton polygon of ${\mathfont{Jac}}(X)$.
A curve $X$ of genus $g$ is \emph{ordinary} if $\nu({\mathfont{Jac}}(X)) = (G_{0,1} \oplus G_{1,0})^g$ and is \emph{supersingular} if 
$\nu({\mathfont{Jac}}(X))=G_{1,1}^g$.  

\subsection{Newton polygons for cyclic covers branched at three points} \label{Scyclic2}

Suppose $\phi : X \to \mathfont{P}^1$ is a cyclic degree $m$ cover which is branched at three points. 
Assuming that $p$ does not divide $m$, we describe how to compute the Newton polygon of 
${\mathfont{Jac}}(X)$ from the signature type of $\phi$, 
using a formula due to Shimura and Taniyama, see \cite[\S 2,3]{LMPT1} for details.  
This implicitly relies on $X$ being the reduction of a curve over $\QQ$ whose Jacobian has complex multiplication.

\begin{notation} \label{NdefR}
Assume that $p$ does not divide $m$.
Let $R$ denote the set of orbits of the multiplication by $p$ map $[\times p]$ on the set $\{j \in \ZZ/m\ZZ : j \neq 0\}$.  
For each orbit $r \in R$, the additive order $d_j$ of $j$ modulo $m$ is the same for all values $j$ in $r$ and we denote it by $d_r$.
For $\epsilon \in \{0,1\}$, let
\begin{equation} \label{EdefS}
S_\epsilon=\{1 \leq j < m : d_j \text{ does not divide }a_b \text{ for } b \in B, \text{ and } f_j=\epsilon\}.
\end{equation}
For $r \in R$ such that $d_r$ does not divide $a_b$ for $b \in B$, 
set $\alpha_r=\#(r \cap S_1)$.
\end{notation}

In the context of Notation~\ref{NdefR}, 
the set $R$ is in one-to-one correspondence with the set of primes of the 
group ring $\QQ[\mu_m]$ above $p$.  This is because $\QQ[\mu_m]$ is unramified at $p$;
the elements of $\ZZ/m\ZZ$ can be identified with homomorphisms from $\QQ[\mu_m]$ to $\mathfont{C}$; and
there is a natural action of Frobenius on the set of homomorphisms, see \cite[\S2.1]{LMPT1} for details.
The $p$-divisible group ${\mathfont{Jac}}(X)[p^\infty]$ canonically decomposes
into pieces ${\mathfont{Jac}}(X)[p^\infty]_r$ indexed by the orbits $r \in R$. 

\begin{thm}[Shimura--Taniyama formula] \label{TSTf} \cite[Section 5]{Tate71}
The only slope of the Newton polygon of ${\mathfont{Jac}}(X)[p^\infty]_r$ is $\lambda_r := \alpha_r/\# r$.
\end{thm}

In particular, if every orbit is self-dual, meaning invariant under negation modulo $m$, then $X$ is supersingular.

\begin{exa} \label{Em8p3}
Let $m=8$ and $p \equiv 3 \bmod 4$.  The orbits are $r_1=\{1,3\}$, $r_2 =\{2,6\}$, $r_5=\{5,7\}$, and $r_4=\{4\}$.

If $a= (1,1,6)$, then $\phi$ has signature $f=(1,1,1,0,0,0,0)$.  
Thus $S_1=\{1,2,3\}$.
Then $\lambda_{r_1} =1$, $\lambda_{r_2} = 1/2$, and $\lambda_{r_5}=0$.
Thus $\nu(\mathfont{Jac}(X)) =(G_{0,1} \oplus G_{1,0})^2 \oplus G_{1,1}$. 

If $a = (1,5,2)$, then $\phi$ has signature $f=(1,1,0,0,1,0,0)$.
Thus $S_1=\{1,2,5\}$ and $X$ is supersingular.
\end{exa}

\begin{rem} There is another method to determine whether $X$ is supersingular by Yui \cite{Yui} and Aoki \cite{Aoki} 
using $p$-adic valuations of Jacobi sums.
By \cite[Theorem 1.1]{Aoki}, $X$ is supersingular if and only if either (i) $p^i \equiv -1 \bmod m$ for some $i$; or
(ii) there is an equivalence (as defined in Section~\ref{ss:types}) of inertia types
$a \approx (1, -p^i, p^i-1)$ for some $i$ such that $d=\mathrm{gcd}(p^i-1, m) > 1$ and $p^j \equiv -1 \bmod m/d$ for some $j$.  
\end{rem}

\subsection{Non-cyclic Galois groups and inertia types}

Suppose $\phi$ is abelian but not cyclic, of order $m$.  The following lemma determines the Galois group of $\phi$ in all the cases with $s>1$ for $5 \leq g \leq 10$.

\begin{lmm} \label{Lnotcyclic}
Let $z =[c_0,c_1,c_\infty,s]$ be a ramification type such that $s >1$, 
$s \mid c_\infty$, $c_\infty \mid c_1$, and $c_1 \mid c_0$.
Then $G$ is non-cyclic with $G \simeq \ZZ/c_0 \ZZ \times \ZZ/s \ZZ$.
\end{lmm}

\begin{proof}
By hypothesis, 
there is a prime $\ell$ dividing all three of $s=c_0s/c_0$, $c_0s/c_1$, and $c_0s/c_\infty$.  
If $G$ was cyclic, $\ell$ would divide each of $m = c_0 s, a_0, a_1, a_\infty$, 
contradicting the fact that $a_0, a_1, a_\infty$ generate $G$.  
Since $G$ is generated by two elements $a_0, a_1$, this implies that 
$G \simeq \ZZ/r_1 \ZZ \times \ZZ/r_2 \ZZ$ where $r_1r_2 = c_0s$ and $r_2 \mid r_1$.
Note that $r_1$ is a multiple of $c_0$ since $a_0$ has order $c_0$.
The hypothesis on $z$ implies that the orders of $a_1$ and $a_\infty$ divide $c_0$.
So $a_0$ and $a_1$ generate $G$ only if $r_1 = c_0$.
Thus $G \simeq \ZZ/c_0 \ZZ \times \ZZ/s \ZZ$.
\end{proof}

For $s >1$, we denote elements of $G = \ZZ/c_0 \ZZ \times \ZZ/s \ZZ$ as an ordered pair.
Each $a_b$ in the inertia type $a=(a_0,a_1,a_\infty)$ is written as $a_b = (a_{b,1}, a_{b,2})$.
To choose a representative for the equivalence class of $a$, 
we take the natural inclusion of $G$ in $V=(\ZZ/c_0 \ZZ)^2$. 
Since $\mathrm{Aut}(V)$ acts transitively
on pairs of linearly independent elements of order $c_0$, 
we may choose $a_0 = (1,0)$ without loss of generality.
When $c_\infty =s$, we also choose $a_\infty =(0,1)$.

\subsection{Decomposition from the Kani--Rosen theorem} \label{Skaniroshen}

Suppose $\phi: X \to \mathfont{P}^1$ is an abelian non-cyclic cover branched at $B=\{0,1,\infty\}$.
It is possible to generalize Sections~\ref{Scyclic1} and \ref{Scyclic2} to the 
abelian non-cyclic case.  This would take us too far afield so we use the Kani--Rosen theorem.

The goal is to decompose $\mathrm{Jac}(X)$, up to isogeny, as the product 
$\prod_{i \in I} J_i$ of abelian varieties $J_i$ of smaller dimension.  
Each abelian variety $J_i$ in this decomposition is the Jacobian of a curve $Z_i$, such that $Z_i$ is a quotient of $X$ 
by a subgroup $H_i \subset G$.  

\begin{thm}[{Kani--Rosen \cite[Theorem~B]{kanirosen}}] \label{TKR} 
Let $G, H_1, \ldots, H_t \subset \mathrm{Aut}(X)$ be such that $G =\cup_{i=1}^t H_i$ and 
$H_i \cap H_j = \{\mathrm{id}\}$ for $i \not = j$.  Let $n_G = \#G$ and $n_i = \#H_i$ for $1 \leq i \leq t$.
Then there is an isogeny:
\begin{equation} \label{Ekanirosen}
\mathrm{Jac}(X)^{t-1} \times \mathrm{Jac}(X/G)^{n_G} \sim \prod_{i=1}^t \mathrm{Jac}(X/H_i)^{n_i}. 
\end{equation}
\end{thm}

\begin{exa}
Suppose $G \simeq (\ZZ/2\ZZ)^2$ and let $H_1, H_2, H_3 \subset G$ be the three subgroups of order two.  
Then \eqref{Ekanirosen} simplifies to $\mathrm{Jac}(X) \times \mathrm{Jac}(X/G)^2 \sim \prod_{i=1}^3\mathrm{Jac}(X/H_i)$.
Suppose $G \simeq (\ZZ/3\ZZ)^2$ and let $H_1, H_2, H_3, H_4 \subset G$ be the four subgroups of 
order three.
Then \eqref{Ekanirosen} simplifies to $\mathrm{Jac}(X) \times \mathrm{Jac}(X/G)^3 \sim \prod_{i=1}^4\mathrm{Jac}(X/H_i)$.
\end{exa}

\subsection{Enumerating and analyzing abelian covers}

Using SageMath, we automated the search for abelian covers $\phi: X \to \mathfont{P}^1$ over $\QQ$ branched at $B=\{0,1,\infty\}$ whose genus lies in a specified range.  
We then partially automated the determination of the set of primes $p$ for which the reduction of $X$ modulo $p$ is supersingular.
We outline the approach here: the code is on github \cite{github}.

We begin by listing all possible $e_{0,1},e_{0,\infty}, e_{1,\infty}, r, s$ which satisfy equation \eqref{Eformulag} for $g$ in the specified range.  These determine the potential ramification types of covers. 
  
For a potential ramification type with $s=1$, the cover $\phi$ is cyclic and $G = \ZZ/m\ZZ$.  
We look for inertia types compatible with the ramification type.
Without loss of generality, $a_0 = 1$, and we search for $a_1$ and $a_\infty$ with $|a_1|=c_1$, $|a_\infty| = c_\infty$, and $a_0 + a_1 + a_\infty = 0 \mod{m}$.  Then \eqref{eq:cycliccover} gives a cover with that inertia type.
Finally, the program determines the congruence classes of primes $p$ for which the reduction of $X$ modulo $p$
is supersingular, 
using the Shimura--Taniyama formula found in Theorem~\ref{TSTf}.
  
When $s >1$, Lemma~\ref{Lnotcyclic} implies that $\phi$ is non-cyclic 
and that $G \simeq (\ZZ/c_0 \ZZ) \times (\ZZ/s \ZZ)$.
Again, we look for inertia types compatible with the ramification type.
Without loss of generality, we identify the potential generator of $I_0$ with $a_0 =(1,0)$ and search for $a_1$ and $a_\infty$ satisfying the constraints listed after Definition~\ref{dfn:inertialtype}.  
Note that an inertia type realizing a potential ramification type may not exist, in which case there is no 
cover $\phi$ realizing that ramification type; see Remark~\ref{Rnovalid} for examples.

In the small number of non-cyclic cases with $5 \leq g \leq 10$, for a curve $X$ with inertia type $a$, we determine by hand 
a decomposition of $\mathfont{Jac}(X)$ into 
Jacobians of curves that are cyclic covers of $\mathfont{P}^1$ branched at $B$, 
using the Kani--Rosen method as in Section~\ref{Skaniroshen}.  
The program then gives us the primes for which the reduction of $X$ is supersingular.

In Section~\ref{Searlier}, we give a complete analysis of the genus $5$ situation to illustrate the methods.  
Section~\ref{Sdatacyclic} contains the results of our program for cyclic cases when $6 \leq g \leq 10$.
Section~\ref{Sdatanoncyclic} treats the non-cyclic cases when $6 \leq g \leq 10$.  

\section{Refining Ekedahl's Results for genus five} \label{Searlier} \label{Section4}

Suppose $\phi: X \to \mathfont{P}^1$ is a tamely ramified abelian cover over $k$ branched at $B=\{0,1,\infty\}$.
In \cite[page~173]{Ekedahl87}, Ekedahl enumerates ways $\phi$ can ramify so that $X$ has genus $5$.
There are seven non-equivalent cases, as confirmed by our program, which
are listed in Propositions~\ref{Pekedahl5cyclic}, \ref{Pekedahl5notcyclica}, 
and \ref{Pekedahl5notcyclicb}. 

Ekedahl studies the situation when $X$ is superspecial, meaning that $\mathrm{Jac}(X)$ is 
isomorphic to a product of supersingular elliptic curves.  The superspecial condition is equivalent to the Frobenius operator being trivial on $H^1(X, \mathcal{O}_X)$.  
For each of the seven cases, 
Ekedahl determines the primes $p$ for which $X$ is superspecial, thus
proving that there exists a superspecial smooth curve of genus $5$ which is an abelian cover of $\mathfont{P}^1$ 
branched at $3$ points if and only if 
$p \equiv -1 \bmod 8, 11,12,15,20$ or $p \equiv 11 \bmod 15$.

If $X$ is superspecial, then it is supersingular, but the converse 
is not always true.
To determine the (less restrictive) conditions on $p$ for which $X$ is 
supersingular, we use the Shimura--Taniyama formula.

\subsection{Cyclic cases for genus five}

We use the abbreviation QNR for quadratic non-residue.

\begin{prp} \label{Pekedahl5cyclic}
If $\phi: X \to \mathfont{P}^1$ is a cyclic cover branched at $3$ points
and $X$ has genus $5$, then its ramification type $z$ and set of supersingular primes are one of:
\[
\begin{tabular} {|l|l|l|}
\hline
Case & z  & supersingular primes \\
\hline
$1$ & $[11,11,11,1]$ & $p \equiv \text{QNR} \bmod 11$ \\
\hline
$2$ & $[15,15,3,1]$ & $p \equiv -1, 11 \bmod 15$ \\
\hline
$3$ & $[22,11,2,1]$ & $p \equiv \text{QNR} \bmod 11$ \\
\hline
$4$ & $[12,12,6,1]$ & $p \equiv -1 \bmod 12$ \\
\hline
$5$ & $[20,20,2,1]$ & $p \equiv -1, 11 \bmod 20$ \\
 \hline
\end{tabular}
\]
\end{prp}

\begin{proof}
By \cite[page~173]{Ekedahl87}, these are the only options for $z$ when $\phi$ is cyclic and $X$ has genus $5$.  We analyze these covers the same way our program does to illustrate how it works.
In each case, $G= \ZZ/m \ZZ$ with $m=c_0$ and $\phi$ is totally ramified over $b=0$.
A short calculation determines the choices for the inertia type $a$, with $a_0 = 1$. 

Recall that $R$ denotes the set of orbits of $\times p$ on $\ZZ/m\ZZ$.
If $p^i \equiv -1 \bmod m$ for some $i$, then each orbit is self-dual and so $X$ is supersingular.
When $m$ is composite, there can be additional cases where $\alpha_r/\#r = 1/2$ for each $r \in R$.  
Here are some details.

\begin{enumerate}
\item %Note that $a \approx (1,k,-(1+k))$ for some $1 \leq k \leq 9$.  
This is one case of \cite[Theorem~5.1]{LMPT1}.

\item 
%$a=(1,k, -(1+k))$ where $\mathrm{gcd}(k,15)=1$ and $\mathrm{gcd}(k+1,15)=5$.  
Note that $a \approx (1,4,10)$.
Then $S_1 = \{1,2,4,5,8\}$.  
When $p \equiv 11 \bmod 15$, then $X$ is supersingular.
%This can also be verified using Aoki's result by noting that $(1,4,10) \sim (1, -p, p -1)$ where $d=\mathrm{gcd}(p-1, 15)=5>1$
%and $p \equiv -1 \bmod m/d$.
This can be verified using the Shimura--Taniyama formula because $\#r = 2$ for each $r \in R$ 
and each orbit contains exactly one element of $S_1$.  
The only other case when $X$ is supersingular is when $p \equiv -1 \bmod{15}$.

\item Then $a \approx (1,10, 11)$.  The condition $p^i \equiv -1 \bmod 22$ is solvable for some $i$ 
if and only if $p$ is a quadratic non-residue modulo $11$.

\item Then $a \approx (1,1,10)$ and $S_1=\{1,2,3,4,5\}$.
The only case when $X$ is supersingular is when $p \equiv -1 \bmod 12$.

 \item Then $a \approx (1,9,10)$ and $S_1 = \{1,3,5,7,9\}$.
When $p \equiv 11 \bmod 20$, then $\#r = 2$ for each $r \in R$ 
and each orbit contains exactly one element of $S_1$. 
By the Shimura--Taniyama formula, $X$ is supersingular (but not superspecial).
%This can also be verified with Aoki's result by noting that $(1,9,10) \sim (1, -p, p -1)$ where $d=\mathrm{gcd}(p-1, 20)=10>1$
%and $p \equiv -1 \bmod m/d$.
\end{enumerate}
\end{proof}

\subsection{Non-cyclic cases for genus five} \label{Snoncyclic5}

Suppose $\phi: X \to \mathfont{P}^1$ is an abelian non-cyclic cover over $k$
branched at $B=\{0,1,\infty\}$.  When $X$ has genus $5$,
Ekedahl shows there are two options for the ramification type: either (6) $z=[8,8,4,2]$ or (7) $z=[12,12,2,2]$.

\subsubsection{Case (6)}
When $z=[8,8,4,2]$, then $G$ has order $16$.
By Lemma~\ref{Lnotcyclic}, $G \simeq \ZZ/8 \ZZ \times \ZZ/2 \ZZ$.
Let $H \subset G$ be the subgroup generated by $(4,0)$ and $(0,1)$. 
Let $E_{1728}$ be the elliptic curve with $j$-invariant $1728$, with 
affine equation $y^2=x^3-x$.

\begin{lmm} \label{Lekedahl5notcyclica}
When $z=[8,8,4,2]$, then $X/H$ is isomorphic to $E_{1728}$.  Also:
\begin{enumerate}
%\item the quotient of $X$ by $H$ is isomorphic to $E_{1728}$;
\item the $H$-Galois subcover $X \to E_{1728}$ factors through three curves $E'$, $W$ and $W'$, 
where $E' \to E_{1728}$ is a $2$-isogeny, and $W$ and $W'$ have genus $3$; and
\item there is an isogeny ${\mathfont{Jac}}(X) \times {\mathfont{Jac}}(E_{1728}) \sim {\mathfont{Jac}}(W) \times {\mathfont{Jac}}(W')$.
\end{enumerate}
\end{lmm}

\begin{proof}
When $z=[8,8,4,2]$, then $G \simeq \ZZ/8 \ZZ \times \ZZ/2 \ZZ$.
After modifying by an element of $\mathfont{Aut}(G)$, the
inertia type can be chosen to be $a=((1,0), (1,1), (6,1))$.
%To see this, without loss of generality, we can suppose that $(1,0)$ is the generator of $I_0$;
%then a generator of $I_1$ has the form $(r,1)$ for some odd integer $r$; 
%the condition that $-(r+1, 1)$ has order $4$ implies that $r=5$ or $r=1$. 
%These two inertia types are equivalent.

The $G$-cover $X \to {\mathfont P}^1$ has:
$2$ points above $0$ with $I_0 = \langle (1,0) \rangle$;
$2$ points above $1$ with $I_1 = \langle (1,1) \rangle$; and
$4$ points above $\infty$ with $I_\infty = \langle (6, 1) \rangle$.

Note that $G/H \simeq \ZZ/4 \ZZ$.
So $\psi: X/H \to {\mathfont{P}}^1$ is a $\ZZ/4\ZZ$-cover branched at $B$.
The intersection of $H$ with each of $I_0$, $I_1$, and $I_\infty$ has order $2$, so
%So there are respectively $1$, $1$, and $2$ points of $X/H$ above $0$, $1$, and $\infty$.  
%By the Riemann--Hurwitz formula, 
$X/H$ has genus $1$.  Since $\psi$ is totally ramified above $0$ and $1$,
this implies that $X/H \cong E_{1728}$.

\begin{enumerate}
\item Let $H_1:=\langle (4,0) \rangle$.
Since $(4,0)$ is in the inertia group at every point of $X$ above $0$, $1$, and $\infty$, 
the double cover $X \to X/{H_1}$ is branched at $8$ points.
So $X/{H_1}$ is an elliptic curve, which we denote $E'$.  
The double cover $E' \to E_{1728}$ is unramified and is thus a $2$-isogeny.
Let $W$ be the quotient of $X$ by $H_2:=\langle (0,1) \rangle$.
Since $H_2 \cap I_b$ is trivial for each $b \in B$, 
the double cover $X \to W$ is unramified and $W$ has genus $3$.
The same holds for the quotient $W'$ of $X$ by $H_3:=\langle (4,1) \rangle$.

\item By Theorem~\ref{TKR},
${\mathfont{Jac}}(X) \times {\mathfont{Jac}}(E_{1728})^2 \sim {\mathfont{Jac}}(E') \times {\mathfont{Jac}}(W) \times {\mathfont{Jac}}(W')$. 
The result follows since the category of abelian varieties up to isogeny is semi-simple, and ${\mathfont{Jac}}(E_{1728}) \sim {\mathfont{Jac}}(E')$.
\end{enumerate}
\end{proof}

\begin{prp} \label{Pekedahl5notcyclica}
Suppose $z=[8,8,4,2]$.  Then $X$ is supersingular if and only if $p \equiv 7 \bmod 8$.
If $p \equiv 3 \bmod 8$, then $\nu({\mathfont{Jac}}(X))=(G_{0,1} \oplus G_{1,0})^2 \oplus G_{1,1}^3$.
%the $p$-divisible group of ${\mathfont{Jac}}(X)$ is isogenous to $(G_{0,1} \oplus G_{1,0})^2 \oplus G_{1,1}^3$.
\end{prp}

\begin{proof}
By Lemma~\ref{Lekedahl5notcyclica}, $E_{1728}$ is a quotient of $X$. 
We restrict to the case $p \equiv 3 \bmod 4$ which is equivalent to $E_{1728} \sim E'$ being supersingular. 
In this case, $X$ is supersingular if and only if
$W$ and $W'$ are both supersingular, by Lemma~\ref{Lekedahl5notcyclica}.

The cover $W \to \mathfont{P}^1$ has Galois group $\ZZ/8\ZZ$ and is branched above $B$ with inertia type 
$\overline{a}=(1,1,6)$.
The cover $W' \to \mathfont{P}^1$ has Galois group $\overline{G}=G/\langle (4,1) \rangle \simeq \ZZ/8 \ZZ$, 
which is generated by $z = (1,1)$.
We compute its inertia type $\overline{a}' =(5z, z, 2z) \approx (1,5,2)$.

If $p \equiv 7 \bmod 8$, then every orbit of $\times p$ on $\{1 \leq j < 8\}$ 
is self-dual, and so both $W$ and $W'$ are supersingular by Theorem~\ref{TSTf}.
Suppose $p \equiv 3 \bmod 8$.  By Example~\ref{Em8p3}, 
the curve $W'$ is supersingular since $\overline{a}' \approx (1,5,2)$, 
while $\nu(\mathfont{Jac}(W)) = (G_{0,1} \oplus G_{1,0})^2 \oplus G_{1,1}$ since $\overline{a}=(1,1,6)$. 
Thus ${\mathfont{Jac}}(X) \times {\mathfont{Jac}}(E_{1728})$ has two slopes of $0$ and $1$, and
$\nu({\mathfont{Jac}}(X))=(G_{0,1} \oplus G_{1,0})^2 \oplus G_{1,1}^3$.
\end{proof}

\begin{rem}
Another construction of supersingular curves of genus $5$ when $p \equiv 7 \bmod 8$ and $p \gg 0$
can be found in \cite[Theorem~1.2]{LMPT2}. 
\end{rem}

\subsubsection{Case (7)}

In case (7), $z=[12,12,2,2]$.  
Recall that $E_{1728}$ is the elliptic curve with $j$-invariant $1728$.
Let $E_{0}$ be the elliptic curve with $j$-invariant $0$, with affine equation $y^2=x^3-1$.

\begin{lmm} \label{Lekedahl5notcyclicb}
When $z=[12,12,2,2]$, then:
\begin{enumerate}
\item the curve $X$ is a Klein-four cover of $\mathfont{P}^1$, whose degree $2$ quotients 
$Z$, $W$, and $P$ have genera $3$, $2$, and $0$ respectively, and there is an isogeny 
${\mathfont{Jac}}(X) \sim {\mathfont{Jac}}(Z) \times {\mathfont{Jac}}(W)$;
\item the curve $Z$ is a $\ZZ/12 \ZZ$-cover of $\mathfont{P}^1$ branched at three points, 
with inertia type $(1, 5, 6)$ and $Z$ is a $\ZZ/3\ZZ$-cover of the elliptic curve $E_{1728}$;
\item the curve $W$ is a Klein-four cover of $\mathfont{P}^1$, whose three degree $2$ quotients 
are isomorphic to $E_0$, $E_0$, and $\mathfont{P}^1$, and 
there is an isogeny ${\mathfont{Jac}}(W) \sim E_0^2$.
\end{enumerate}
\end{lmm}

\begin{proof}
By Lemma~\ref{Lnotcyclic}, $G \simeq \ZZ/12 \ZZ \times \ZZ/2 \ZZ$.
After modifying by an element of $\mathfont{Aut}(G)$, we can suppose that $a=((1,0), (11,1), (0,1))$. 
%The curve $X$ has two points above $0$ with inertia group $I_0 = \langle (1,0) \rangle$, 
%two points above $1$ with inertia group $I_1 = \langle (11,1) \rangle$, and 
%twelve points above $\infty$ with inertia group $I_\infty = \langle (0,1) \rangle$.
\begin{enumerate}

\item
The automorphism $(0,1)$ fixes the twelve points of $X$ above $\infty$.
So the quotient $P$ of $X$ by $I_\infty$ is $\mathfont{P}^1$.
The automorphism $(6,0)$ fixes the two points of $X$ above $0$ and the two points of $X$ above $1$.
So the quotient $W$ of $X$ by $\langle (6,0) \rangle$ has genus $2$.
The automorphism $(6,1)$ fixes no points of $X$.
So the quotient $Z$ of $X$ by $\langle (6,1) \rangle$ has genus $3$.
By Theorem~\ref{TKR}, ${\mathfont{Jac}}(X) \sim {\mathfont{Jac}}(Z) \times {\mathfont{Jac}}(W)$. 

\item The curve $Z$ has an action by $G/\langle (6,1) \rangle \simeq \ZZ/12\ZZ$, 
which is generated by $w=(1,1)$.
The inertia type $a$ reduces to $\overline{a} = (7w, 11w, 6w) \approx (5,1,6)$.
The curve $Z$ has an order $3$ automorphism $\sigma$ that fixes two points and $Z/\langle \sigma \rangle$ is 
an elliptic curve with an automorphism of order $4$ that fixes two points.  Thus
$Z$ is a $\ZZ/3\ZZ$-cover of $E_{1728}$.

\item The curve $W$ admits an action of $\ZZ/6\ZZ \times \ZZ/2 \ZZ$.
The automorphisms $(3,0)$ and $(3,1)$ each fix two points of $W$, 
so $W$ admits two double covers of an elliptic curve, which must be $E_0$
because it has an automorphism of order $3$ which fixes three points.
Since $(3,0)$ and $(3,1)$ generate a Klein-four group, 
${\mathfont{Jac}}(W) \sim E_0^2$ by Theorem~\ref{TKR}.
\end{enumerate}
\end{proof}

\begin{prp} \label{Pekedahl5notcyclicb}
Suppose $z=[12,12,2,2]$. 
Then $X$ is supersingular if and only if $p \equiv 11 \bmod 12$.
If $p \equiv 5 \bmod 12$, then $\nu({\mathfont{Jac}}(X))=(G_{0,1} \oplus G_{1,0})^3 \oplus G_{1,1}^2$.
If $p \equiv 7 \bmod 12$, then $\nu({\mathfont{Jac}}(X))=(G_{0,1} \oplus G_{1,0})^2 \oplus G_{1,1}^3$.
\end{prp}

\begin{proof}
Note that $E_0$ is supersingular exactly when $p \equiv 2 \bmod 3$
and $E_{1728}$ is supersingular exactly when $p \equiv 3 \bmod 4$. 
By Lemma~\ref{Lekedahl5notcyclicb}, $X$ dominates $E_0$ and $E_{1728}$.
Thus $X$ supersingular implies that $p \equiv 11 \bmod 12$.

Suppose $p \equiv 11 \bmod 12$.  Then $E_0$ is supersingular.
By Lemma~\ref{Lekedahl5notcyclicb}, ${\mathfont{Jac}}(X) \sim {\mathfont{Jac}}(Z) \times E_0^2$.
The curve $Z$ is supersingular by Theorem~\ref{TSTf}, since the orbits of $\times p$ on $\ZZ/12 \ZZ$ are 
self-dual.  This completes the proof of the first claim.

The $\ZZ/12 \ZZ$-cover $Z \to \mathfont{P}^1$ has inertia type $(5,1,6)$.  We
compute that the eigenspaces of $H^0(Z, \Omega^1_Z)$ of positive dimension are indexed by 
$S_1 = \{1,3,5\}$.  
When $p \equiv 5 \bmod 12$, then $E_0$ is supersingular.
Also, the orbits of $\times p$ are $\{1,5\}$, $\{7,11\}$, 
$\{3\}$ and $\{9\}$.  By Theorem~\ref{TSTf}, $\nu(\mathfont{Jac}(Z)) = (G_{0,1} \oplus G_{1,0})^3$.
When $p \equiv 7 \bmod 12$, then $E_0$ is ordinary.  Also, the orbits of $\times p$ are $\{1,7\}$, $\{5,11\}$, 
and $\{3, 9\}$.  By Theorem~\ref{TSTf}, $Z$ is supersingular.  \end{proof}

\section{Cyclic Examples Genus 6 through 10} \label{Sdatacyclic}

Our program computed this data for curves $X$ over $k$, of genus $6 \leq g \leq 10$, 
that occur as a degree $m$ cyclic cover $\phi: X \to \mathfont{P}^1$ branched at three points: the ramification types $z$; 
the inertia types $a$, up to equivalence; and the odd primes $p$ with $\mathrm{gcd}(p,m)=1$ for which the reduction of 
$X$ modulo $p$ is supersingular.  The results are in Tables~\ref{genus6}-\ref{genus10}.

%\begin{table}[htbp]
%\begin{tabular}{|l|l|l|} \hline
%$z$ & $a$ & supersingular primes \\ \hline
%$\left[11, 11, 11, 1\right]$ & $\left(1, 9, 1\right)$ & $p \equiv QNR \bmod{11}$ \\ \hline
% & $\left(1, 8, 2\right)$ & $p \equiv QNR \bmod{11}$ \\ \hline
%$\left[12, 12, 6, 1\right]$ & $\left(1, 1, 10\right)$ & $p \equiv 11 \bmod{12}$ \\ \hline
%$\left[15, 15, 3, 1\right]$ & $\left(1, 4, 10\right)$ & $p \equiv 11, 14 \bmod{15}$ \\ \hline
%$\left[20, 20, 2, 1\right]$ & $\left(1, 9, 10\right)$ & $p \equiv 11, 19 \bmod{20}$ \\ \hline
%$\left[22, 11, 2, 1\right]$ & $\left(1, 10, 11\right)$ & $p \equiv QNR \bmod{22}$ \\ \hline
%\end{tabular}
%\caption{Genus $5$}
%\end{table}

\begin{table}[htbp]
\centering
\begin{tabular}{|l|l|l|} \hline
$z$ & $a$ & supersingular primes \\ \hline
$\left[13, 13, 13, 1\right]$ & $\left(1, 11, 1\right), \left(1, 10, 2\right), \left(1, 9, 3\right) $ & $p \not \equiv 1,3,9 \bmod{13}$ \\ \hline
 
$\left[14, 14, 7, 1\right]$ & $\left(1, 11, 2\right), \left(1, 5, 8\right), \left(1, 1, 12\right)$ & $p \equiv QNR \bmod{14}$ \\ \hline
$\left[15, 15, 5, 1\right]$ & $\left(1, 11, 3\right)$ & $p \equiv 2,3,4 \bmod 5$ \\ 
%$p \equiv 2, 4, 7, 8, 13, 14 \bmod{15}$ \\ \hline
 & $\left(1, 8, 6\right)$ & $p \equiv 7, 13, 14 \bmod{15}$ \\ \hline
$\left[16, 16, 4, 1\right]$ & $\left(1, 11, 4\right)$ & $p \equiv 15 \bmod{16}$ \\ \hline
$\left[20, 5, 4, 1\right]$ & $\left(1, 4, 15\right)$ & $p \equiv 19 \bmod{20}$ \\ \hline
$\left[18, 18, 3, 1\right]$ & $\left(1, 11, 6\right)$ & $p \equiv 2 \bmod{3}$ \\ \hline
%$p \equiv QNR \bmod{18}$
$\left[21, 7, 3, 1\right]$ & $\left(1, 6, 14\right)$ & $p \equiv 5, 17, 20 \bmod{21}$ \\ \hline
$\left[24, 24, 2, 1\right]$ & $\left(1, 11, 12\right)$ & $p \equiv 13, 23 \bmod{24}$ \\ \hline
$\left[26, 13, 2, 1\right]$ & $\left(1, 12, 13\right)$ & $p \not \equiv 1,3,9 \bmod{26}$ \\ \hline
%$p \equiv 5, 7, 11, 15, 17, 19, 21, 23, 25 \bmod{26}$ \\ \hline
\end{tabular}
\caption{Cyclic genus $6$} \label{genus6}
\end{table}

\begin{table}[htbp]
\centering
\begin{tabular}{|l|l|l|} \hline
$z$ & $a$ & supersingular primes \\ \hline
$\left[15, 15, 15, 1\right]$ & $\left(1, 13, 1\right)$ & $p \equiv 14 \bmod{15}$ \\ \hline
$\left[16, 16, 8, 1\right]$ & $\left(1, 13, 2\right), \left(1, 1, 14\right)$ & $p \equiv 15 \bmod{16}$ \\ 
 & $\left(1, 9, 6\right)$ & $p \equiv 7, 15 \bmod{16}$ \\ \hline
$\left[18, 9, 6, 1\right]$ & $\left(1, 14, 3\right), \left(1, 2, 15\right)$ & $p \equiv 2 \bmod{3}$ \\ \hline
%$p \equiv QNR \bmod{18}$
$\left[20, 10, 4, 1\right]$ & $\left(1, 14, 5\right)$ & $p \equiv 19 \bmod{20}$ \\ \hline
$\left[21, 21, 3, 1\right]$ & $\left(1, 13, 7\right)$ & $p \equiv 2 \bmod{3}$ \\ \hline
%$p \equiv 2, 5, 8, 11, 17, 20 \bmod{21}$
$\left[24, 8, 3, 1\right]$ & $\left(1, 15, 8\right)$ & $p \equiv 23 \bmod{24}$ \\ \hline
$\left[28, 28, 2, 1\right]$ & $\left(1, 13, 14\right)$ & $p \equiv 3 \bmod{4}$ \\ \hline
%$p \equiv 3, 11, 15, 19, 23, 27
$\left[30, 15, 2, 1\right]$ & $\left(1, 14, 15\right)$ & $p \equiv 29 \bmod{30}$ \\ \hline
\end{tabular}
\caption{Cyclic genus $7$}
\end{table}

\begin{table}[htbp]
\centering
\begin{tabular}{|l|l|l|} \hline
$z$ & $a$ & supersingular primes \\ \hline
$\left[17, 17, 17, 1\right]$ & $\left(1, 15, 1\right), \left(1, 14, 2\right), \left(1, 13, 3\right)$ & $p \not \equiv 1 \bmod{17}$ \\ \hline
$\left[18, 18, 9, 1\right]$ & $\left(1, 13, 4\right),\left(1, 1, 16\right)$ & $p \equiv 2 \bmod{3}$ \\ \hline
%$p \equiv QNR \bmod{18}$
$\left[20, 20, 5, 1\right]$ & $\left(1, 11, 8\right)$ & $p \equiv 2,3,4 \bmod{5}$ \\ 
%$p \equiv 3, 7, 9, 13, 17, 19 \bmod{20}$
 & $\left(1, 7, 12\right)$ & $p \equiv 13, 17, 19 \bmod{20}$ \\ \hline
$\left[24, 24, 3, 1\right]$ & $\left(1, 7, 16\right)$ & $p \equiv 17, 23 \bmod{24}$ \\ \hline
$\left[32, 32, 2, 1\right]$ & $\left(1, 15, 16\right)$ & $p \not \equiv 1,15 \bmod{32}$ \\ \hline
$\left[34, 17, 2, 1\right]$ & $\left(1, 16, 17\right)$ & $p \not \equiv 1 \bmod{34}$ \\ \hline
\end{tabular}
\caption{Cyclic genus $8$}
\end{table}

\begin{table}[htbp]
\centering
\begin{tabular}{|l|l|l|} \hline
$z$ & $a$ & supersingular primes \\ \hline
$\left[19, 19, 19, 1\right]$ & $\left(1, 17, 1\right),\left(1, 16, 2\right)$, & $p \equiv QNR \bmod{19}$ \\
& $\left(1, 15, 3\right), \left(1, 11, 7\right)$ & $p \equiv QNR \bmod{19}$ \\ \hline
$\left[20, 20, 10, 1\right]$ & $\left(1, 17, 2\right)$ & $p \equiv 3, 7, 19 \bmod{20}$ \\ \hline
 & $\left(1, 1, 18\right)$ & $p \equiv 19 \bmod{20}$ \\ \hline
$\left[21, 21, 7, 1\right]$ & $\left(1, 17, 3\right), \left(1, 11, 9\right)$ & $p \equiv 5, 10, 17, 19, 20 \bmod{21}$ \\ 
  & $\left(1, 8, 12\right)$ & $p \equiv QNR \bmod{7}$ \\ \hline
%$p \equiv 5, 10, 13, 17, 19, 20 \bmod{21}$
$\left[24, 8, 6, 1\right]$ & $\left(1, 3, 20\right)$ & $p \equiv 23 \bmod{24}$ \\ \hline
$\left[24, 24, 4, 1\right]$ & $\left(1, 17, 6\right)$ & $p \equiv 7, 23 \bmod{24}$ \\ \hline
 & $\left(1, 5, 18\right)$ & $p \equiv 19, 23 \bmod{24}$ \\ \hline
$\left[28, 7, 4, 1\right]$ & $\left(1, 20, 7\right)$ & $p \equiv 3, 19, 27 \bmod{28}$ \\ \hline
$\left[27, 27, 3, 1\right]$ & $\left(1, 17, 9\right)$ & $p \equiv 2 \bmod{3}$ \\ \hline
%$p \equiv QNR \bmod{27}$
$\left[30, 10, 3, 1\right]$ & $\left(1, 9, 20\right)$ & $p \equiv 29 \bmod{30}$ \\ \hline
$\left[36, 36, 2, 1\right]$ & $\left(1, 17, 18\right)$ & $p \equiv 3 \bmod{4}$ \\ \hline
%$p \equiv 7, 11, 19, 23, 31, 35 \bmod{36}$
$\left[38, 19, 2, 1\right]$ & $\left(1, 18, 19\right)$ & $p \equiv QNR \bmod{38}$ \\ \hline
\end{tabular}
\caption{Cyclic genus $9$}
\end{table}

\begin{table}[htbp]
\centering
\begin{tabular}{|l|l|l|} \hline
$z$ & $a$ & supersingular primes \\ \hline
$\left[21, 21, 21, 1\right]$ & $\left(1, 19, 1\right), \left(1, 16, 4\right)$ & $p \equiv 5, 17, 20 \bmod{21}$ \\ \hline
$\left[22, 22, 11, 1\right]$ & $\left(1, 19, 2\right), \left(1, 17, 4\right),\left(1, 15, 6\right) $ & $p \equiv QNR \bmod{22}$ \\ 
 & $\left(1, 9, 12\right),\left(1, 1, 20\right)$ & $p \equiv QNR \bmod{22}$ \\ \hline
$\left[24, 12, 8, 1\right]$ & $\left(1, 14, 9\right), \left(1, 2, 21\right)$ & $p \equiv 23 \bmod{24}$ \\ \hline
$\left[24, 24, 6, 1\right]$ & $\left(1, 19, 4\right)$ & $p \equiv 5, 23 \bmod{24}$ \\ \hline
$\left[25, 25, 5, 1\right]$ & $\left(1, 19, 5\right),\left(1, 14, 10\right)$ & $p \neq 1 \bmod{5}$ \\ \hline
$\left[30, 6, 5, 1\right]$ & $\left(1, 5, 24\right)$ & $p \equiv 29 \bmod{30}$ \\ \hline
$\left[28, 14, 4, 1\right]$ & $\left(1, 6, 21\right)$ & $p \equiv 3, 19, 27 \bmod{28}$ \\ \hline
$\left[30, 30, 3, 1\right]$ & $\left(1, 19, 10\right)$ & $p \equiv 11, 29 \bmod{30}$ \\ \hline
$\left[33, 11, 3, 1\right]$ & $\left(1, 21, 11\right)$ & $p \equiv 2, 8, 17, 29, 32 \bmod{33}$ \\ \hline
$\left[40, 40, 2, 1\right]$ & $\left(1, 19, 20\right)$ & $p \equiv 21, 39 \bmod{40}$ \\ \hline
$\left[42, 21, 2, 1\right]$ & $\left(1, 20, 21\right)$ & $p \equiv 5, 17, 41 \bmod{42}$ \\ \hline
\end{tabular}
\caption{Cyclic genus $10$} \label{genus10}
\end{table}

\section{Non-Cyclic Examples} \label{Sdatanoncyclic}

Our program identified $14$ ramification types $z=[c_0, c_1,c_\infty, s]$, with $s >1$, for which there is a valid inertia type, 
for an abelian cover $\phi: X \to \mathfont{P}^1$ branched at three points such that the genus $g$ of $X$ satisfies $6 \leq g \leq 10$.
In this section, we analyze the supersingular primes $p$ for these covers, assuming throughout that $p \nmid c_0s$.
We also prove a general result about a class of examples which occur for every genus. 

\begin{rem} \label{Rnovalid}
For $5 \leq g \leq 10$, we encountered three cases with no inertia type realizing the ramification type:
($g=6$) $z = [8,8,8,2]$; ($g=8$) $z =[12, 12,4,2]$; and
($g=10$) $z=[12,12,12,2]$.
\end{rem}

Here is some notation. If $g \in G$ has prime order $\ell$,
let $X_g = X/\langle g \rangle$ be the quotient curve and 
$\phi_g: X \to X_g$ be the quotient cover.  
For $b \in \{0,1,\infty\}$, the inertia group $I_b$ is independent of the ramification point $\eta \in X$ above $b$ 
since $G$ is abelian.
If $g \in I_b$, then $\phi_g$ is ramified with index $\ell$ at $\eta$.
If $g \not \in I_b$, then $\phi_g$ is not ramified at $\eta$.

\subsection{A useful lemma}

\begin{lmm} \label{Lreduce2}
Let $r \geq 2$.
If $z = [2r,2r,2,2]$, then $\mathfont{Jac}(X) \sim \mathfont{Jac}(X_1) \times \mathfont{Jac}(X_2)$, 
where $\phi_i: X_i \to \mathfont{P}^1$ is an abelian cover branched at $B=\{0,1,\infty\}$ for $i=1,2$.
If $r$ is odd and $i=1,2$, or if $r$ is even and $i=1$, 
then $\phi_i$ has group $\ZZ/2r \ZZ$ with $a=(1, r, r-1)$.
If $r$ is even, then 
$\phi_2$ has group $\ZZ/r \ZZ \times \ZZ/2 \ZZ$ with $a=((1,0), (-1,-1), (0,1))$.
\end{lmm}

\begin{proof}
By the Riemann--Hurwitz formula, $X$ has genus $r-1$. 
By Lemma~\ref{Lnotcyclic}, $G = \ZZ/2r \ZZ \times \ZZ/2\ZZ$. 
Without loss of generality, $a_0=(1,0)$, $a_1=(-1,1)$, and $a_\infty = (0,1)$.
Since $a_\infty \in I_\infty$, the double cover
$\phi_{a_\infty} : X \to X_{a_\infty}$ is branched at the $2r$ points above $\infty$.
Thus $X_{a_\infty}$ has genus $0$.  

Let $K=\langle (r,0), (0,1) \rangle$.  Since $X/K$ is a quotient of $X_{a_\infty}$, it also has genus $0$.
Write $X_1 = X/\langle (r,1) \rangle$ and $X_2 = X/\langle (r,0) \rangle$.
By Theorem~\ref{TKR}, 
$\mathfont{Jac}(X) \sim \mathfont{Jac}(X_1) \times \mathfont{Jac}(X_2)$.

For $g=(r,0)$ and $g=(r,1)$, note that $g \not \in I_\infty$; so $\phi_g$ is not ramified in the fiber over $\infty$.
There are $2$ points of $X$ above each of $0$ and $1$.
Now $(r,0) \in I_0$ and $(r,1) \not \in I_0$.
Also $(r,0) \in I_1$ if and only if $r$ is even, while
$(r,1) \in I_1$ if and only if $r$ is odd.

Suppose $r$ is odd.  Then $X_1$ and $X_2$ both have genus $(r-1)/2$.
The quotient group $\overline{G}_2 = G/\langle (r,0) \rangle$ is cyclic of order $2r$ with generator $\tau_0=(1,1)$.
Then $a_1 = - \tau_0$, $a_\infty \equiv r \tau_0 \bmod (r,0)$, and
$a_0 = (r+1) \tau_0$.  Thus the reduction of the inertia type in $\overline{G}_2$ is $(r+1, r, -1) \approx (1, r, r-1)$.
The quotient group $\overline{G}_1 = G/\langle (r,1) \rangle$ is cyclic of order $2r$ with generator $\tau_1=(1,0) = a_0$.
We see that $a_\infty \equiv r \tau_1$ and $a_1 \equiv (r-1) \tau_1$.
Thus the $\ZZ/2r\ZZ$-cover $\phi_i: X_i \to \mathfont{P}^1$
has inertia type $(1, r, r-1)$ for $i=1,2$.

Suppose $r$ is even.
Then $\phi_{(r,1)}$ is unramified and $X_1$ has genus $r/2$;
also $\phi_{(r,0)}$ is ramified at four points and $X_2$ has genus $(r/2)-1$.
Also $\overline{G}_1 = G/\langle (r,1) \rangle$ is cyclic of order $2r$ with generator $\tau_1 = (1,1) = -a_1$.
By a similar argument, $\phi_1$ has inertia type $(1,r,r-1)$.
The inertia type for $\phi_2$ is found by working in $G/\langle (r,0) \rangle$.
\end{proof}

\begin{rem} \label{Rpositivedensity}
By Lemma~\ref{Lreduce2}, for $r \geq 2$, 
the ramification type $z = [2r,2r,2,2]$ yields a curve $X$ of genus $r-1$ which is supersingular for all primes $p \equiv -1 \bmod 2r$. 
This shows that the natural density of $\SSS_g$ is positive.
We omit the details because this can also be proven using \cite{Yui} or (unpublished) work in \cite{surim}.
\end{rem}

\subsection{Non-Cyclic Examples: Genus 6 through 10}

The tables below contain the following information for the $14$ non-cyclic cases for $6 \leq g \leq 10$:
the ramification type $z$; the decomposition of the Jacobian; the congruence conditions on $p$, assuming $\mathrm{gcd}(p, |G|)=1$, 
for which the curve $X$ is supersingular.

The decomposition of the Jacobian is given by the data $(m_i, \alpha_i)$ indexed by $i \in I$, 
meaning that there is an isogeny $\mathfont{Jac}(X) \sim \prod_{i \in I} \mathfont{Jac}(Z_i)$, 
where there is a cyclic degree $m_i$ 
cover $\phi_i: Z_i \to \mathfont{P}^1$ branched at $B =\{0,1,\infty\}$ with inertia type $\alpha_i$.  
We use an addition sign to denote a product of 
Jacobians, an exponent to denote a product with multiplicity, and a subtraction sign to denote a quotient.

\begin{lmm}
Here are the supersingular primes for genus six curves realized as a non-cyclic abelian cover of $\mathfont{P}^1$ 
branched at $B=\{0,1,\infty\}$:
\[
\begin{tabular} {|l|l|l|l|}
\hline
$z$ & %$a$ & 
decomposition & supersingular primes \\
\hline
$[5,5,5,5]$ & %$((1,0), (0,1), (4,4))$ & 
$(5, (1,1,3))^3$ & $p \equiv 2,3,4 \bmod{5}$ \\
\hline
$[14,14,2,2]$ & %$((1,0), (-1,-1), (0,1))$ & 
$(14, (1,6,7))^2$ & $p \equiv 3, 5, 13 \bmod{14}$ \\
\hline
\end{tabular}
\]
\end{lmm}

\begin{proof}
When $g=[5,5,5,5]$, 
then $a_0=(1,0)$, $a_1 = (-1,-1)$, and $a_\infty=(0,1)$.
By Theorem~\ref{TKR}, $\mathfont{Jac}(X)$ is isogenous to the product of the Jacobians 
of six quotients of $X$.
For $g = (1,0)$, $(0,1)$, and $(1,1)$, the cover $\phi_g$ is ramified at the five points above one point $b \in B$; 
thus $X_g$ has genus $0$.
For $g=(2,1)$, $(3,1)$, and $(4,1)$, the cover $\phi_g$ is not ramified and $X_g$ has genus two;
the cover $\psi_g: X_g \to \mathfont{P}^1$ is a $\ZZ/5 \ZZ$-cover 
branched at $B$; up to equivalence, there is only one inertia type for $\psi_g$, namely $(1,1,3)$.

The decomposition for $z=[14,14,2,2]$ follows from Lemma~\ref{Lreduce2}.
\end{proof}

\begin{lmm}
Here are the supersingular primes for genus seven curves realized as a non-cyclic abelian cover of $\mathfont{P}^1$ 
branched at $B=\{0,1,\infty\}$:
\[
\begin{tabular} {|l|l|l|l|}
\hline
$z$ & %$a$ & 
decomposition & supersingular primes \\
\hline
$[9,9,3,3]$ & 
%$((1,0), (-1,-1), (0,1))$& 
$(3, (1,1,1)) + (9, (1,3,5)) + (9, (1,2,6))$ &  $p \equiv 2 \bmod{3}$ \\ \hline 
$[12,6,4,2]$ & %$((1,0), (2,1), (9,1))$  &
$(12, (1,2,9)) + (12, (1,3,8))$ & $p \equiv 11 \bmod{12}$ \\ \hline
$[16,16,2,2]$ & %$((1,0), (-1,-1), (0,1))$ & 
$(16, (1,7,8)) + (8, (1,3,4)) + (4, (1,1,2))$ & $p \equiv 15 \bmod{16}$ \\ \hline
\end{tabular}
\]
\end{lmm}

\begin{proof}
Suppose $z = [9,9,3,3]$.  Then $G=\ZZ/9\ZZ \times \ZZ/3\ZZ$ by Lemma~\ref{Lnotcyclic}.
Then $a_0=(1,0)$, $a_1=(-1,-1)$, and $a_\infty=(0,1)$.
Let $H \subset G$ be generated by $g_0=(3,0)$ and $g_1=(0,1)$.
The cover $\phi_{(0,1)}$ is ramified at the nine points above $\infty$; thus $X_{(0,1)}$ has genus $0$.
By Theorem~\ref{TKR}, $\mathfont{Jac}(X) \sim \mathfont{Jac}(X_{(3,0)}) \times \mathfont{Jac}(X_{(3,1)}) \times \mathfont{Jac}(X_{(3,2)})$. 
The cover $\phi_{(3,0)}$ is ramified at the three points above $0$ and at the three points above $1$.
Thus $X_{(3,0)}$ has genus $1$ and admits a cyclic degree $3$ cover of $\mathfont{P}^1$, 
with inertia type $(1,1,1)$.

The quotient group $G/\langle (3,2) \rangle$ is cyclic of order $9$, generated by $\tau =(1,1)=-a_1$.
Then $a_0 \equiv 7 \tau$ and $a_\infty \equiv 3 \tau$.  So $\phi_{(3,2)}$ has inertia type $(7,3,-1) \approx (1,2,6)$.
The quotient group $G/\langle (3,1) \rangle$ is cyclic of order $9$, generated by $\tau =(1,1)=-a_1$.
Then $a_0 \equiv 4 \tau$ and $a_\infty \equiv 6 \tau$.  So $\phi_{(3,2)}$ has inertia type $(4,6,-1) \approx (1,3,5)$.

Suppose $z = [12,6,4,2]$.  Then $G=\ZZ/12\ZZ \times \ZZ/2\ZZ$.
Then $a_0=(1,0)$, $a_1=(2,1)$, and $a_\infty=(9,1)$.
%There are $2$ points above $b=0$, $4$ points above $b=1$, and $6$ points above $b= \infty$.
Let $K=\langle (6,0), (0,1)\rangle$.  
The quotient $X_{(0,1)}$ has genus $4$ and is a $\ZZ/12\ZZ$ cover with inertia type $(1,2,9)$.
The quotient $X_{(6,1)}$ has genus $3$ and is a $\ZZ/12\ZZ$ cover with inertia type $(1,3,8)$.
The quotient $X_{(6,0)}$ has genus $2$ and is a $\ZZ/6\ZZ \times \ZZ/2 \ZZ$ cover with inertia type 
$((1,0), (2,1), (3,1))$.  
The quotient by $K$ is an elliptic curve $D$ which has an automorphism of order $3$.
By Theorem~\ref{TKR}, $\mathfont{Jac}(X_{(6,0)}) \sim D^2$.
Thus $\mathfont{Jac}(X) \sim \mathfont{Jac}(X_{(0,1)}) \times \mathfont{Jac}(X_{(6,1)})$. 

The case $[16,16,2,2]$ follows from Lemma~\ref{Lreduce2} (used twice).
\end{proof}

\begin{lmm}
Here are the supersingular primes for genus eight curves realized as a non-cyclic abelian cover of $\mathfont{P}^1$ 
branched at $B=\{0,1,\infty\}$: 
\[
\begin{tabular} {|l|l|l|l|}
\hline
$z$ & % $a$ & 
decomposition & supersingular primes \\
\hline
$[10,10,10,2]$ & %$((1,0), (1,1), (8,1))$ & 
$(10, (1,1,8)) + (10, (1,3,6))^2 - (5, (1,1,3))$ & $p \equiv 3,7,9 \bmod{10}$ \\ \hline
$[10,10,10,2]$ & %$((1,0), (2,1), (7,1))$ & 
$(10, (1,1,8)) + (10, (1,2,7))^2 - (5, (1,1,3))$& $p \equiv 3,7,9 \bmod{10}$ \\ \hline
$[18,18,2,2]$ & %$((1,0), (-1,-1), (0,1))$ & 
$(18, (1,8,9))^2$ & $p \equiv 5, 11, 17 \bmod{18}$ \\ \hline
\end{tabular}
\]
\end{lmm}

\begin{proof}
When $z=[10,10,10,2]$, there are two possibilities for $a$, up to equivalence: $a=((1,0), (1,1), (8,1))$
or $a=((1,0), (2,1), (7,1))$.
We apply Theorem~\ref{TKR} using the Klein-four group $K$.  
The quotients by the three involutions are $\ZZ/10\ZZ$-covers of $\mathfont{P}^1$ branched at $B$.
The quotient $X/K$ has genus $2$ and admits a $\ZZ/5\ZZ$-cover of 
$\mathfont{P}^1$ branched at $B$ with $a=(1,1,3)$.  

The case $[18,18,2,2]$ follows from Lemma~\ref{Lreduce2}.
\end{proof}

We determine enough information to show that abelian 
non-cyclic covers produce no new primes for supersingular curves of genus $9$.
%because of the cases $z=(36,36,2,1)$ and $[20,20,10,1]$.

\begin{lmm}
Here are the supersingular primes for genus nine curves realized as a non-cyclic abelian cover of $\mathfont{P}^1$ 
branched at $B=\{0,1,\infty\}$:
\[
\begin{tabular} {|l|l|l|l|}
\hline
$z$  & supersingular primes \\
\hline
$[8,8,4,4]$ & necessary condition $p \equiv 7 \bmod 8$  \\ \hline
$[12, 12, 6, 2]$ & necessary condition $p \equiv 11 \bmod 12$ \\ \hline
$[20,20,2,2]$  & $p \equiv 19 \bmod{20}$ \\ \hline
\end{tabular}
\]
\end{lmm}

\begin{proof}
For $z=[8,8,4,4]$, $G=\ZZ/8 \ZZ \times \ZZ/4 \ZZ$ and
$a=((1,0), ((7,3), (0,1))$.
We consider the quotient by $g=(4,2)$.  Since $(4,2)$ is not in $I_0$, $I_1$, or $I_\infty$, 
the cover $\phi_{(4,2)}$ is unramified.  Thus $X_{(4,2)}$ has genus $5$ and Galois group 
$\overline{G}=(\ZZ/8 \ZZ \times \ZZ/4 \ZZ)/\langle (4,2) \rangle$.  Note that $\overline{G} \simeq \ZZ/8 \ZZ \times \ZZ/2 \ZZ$, 
with generators $\tau=(1,1)$ and $\sigma = (2,1)$.  The inertia type reduces to $7 \tau + \sigma$, $7 \tau$, $2\tau + \sigma$.
After multiplying by $-\tau$ and rearranging, 
this is Case~(6) from $g=5$.  By Proposition~\ref{Pekedahl5notcyclica}, 
$X_{(4,2)}$ is supersingular if and only if $p \equiv 7 \bmod 8$.

For $z = [12, 12, 6, 2]$, the group is $\ZZ/12\ZZ \times \ZZ/2 \ZZ$. 
Then $a=((1,0), (1,1), (10,1))$.
%The automorphism $(4,0)$ is in each of $I_0$, $I_1$, and $I_\infty$.  So the cover $\phi_{(4,0)}$ is unramified of degree $3$.
%Thus $X_{(4,0)}$ has genus $1$ and automorphism group $\ZZ/4 \ZZ \times \ZZ/2 \ZZ$ with inertia type
%$((1,0), (3,1), (2,1))$.  The curve $X_{(4,0)}$ is supersingular if and only if $p \equiv 3 \bmod 4$, which gives no new cases.
The quotient by $(0,1)$ is unramified, so $X_{(0,1)}$ has genus $5$.  It admits a cyclic 
degree 12 cover of $\mathfont{P}^1$ branched at $B$ with inertia type $(1,1,10)$.  
This is Case~(7) from $g=5$.  By Proposition~\ref{Pekedahl5notcyclicb}, 
$X$ is supersingular only when $p \equiv 11 \bmod 12$. 

By Lemma~\ref{Lreduce2}, $[20,20,2,2]$ decomposes as $(20, (1,9,10)) + (10, (1,4,5))^2$.
\end{proof}

\begin{lmm}
Here are the supersingular primes for genus ten curves realized as a non-cyclic abelian cover of $\mathfont{P}^1$ 
branched at $B=\{0,1,\infty\}$:
\[
\begin{tabular} {|l|l|l|l|}
\hline
$z$ & % $a$ & 
decomposition & supersingular primes \\
\hline
$[6,6,6,6]$ & %$((1,0), (0,1), (5,5))$ & 
partial information below & $p \equiv 5 \bmod 6$ \\ \hline
$[9,9,9,3]$ & %$((1,0), (1,1), (7,2))$ & 
$(9, (1,1,7))^4 - (3, (1,1,1))^2$ & $p \equiv 2 \bmod{3}$ \\ \hline
$[12,12,3,3]$ & %$((1,0), (11,2), (0,1))$ & 
$(12, (1,3,8)^2) + (12, (1,4,7))$ & $p \equiv 11 \bmod{12}$ \\ \hline
$[22,22,2,2]$ &  %$((1,0), (-1,-1), (0,1))$ & 
$(22, (1,10,11))^2$ & $p \equiv 7, 13, 17, 19, 21 \bmod{22}$ \\
\hline
\end{tabular}
\]
\end{lmm}

\begin{proof}
If $z = [6,6,6,6]$, then $G= \ZZ/6 \ZZ \times \ZZ/6 \ZZ$ by Lemma~\ref{Lnotcyclic} and we may take
$a=((1,0), (0,1), (5,5))$.  Take $H= \langle (2,0), (0,2) \rangle \simeq (\ZZ/3 \ZZ)^2$.
For $g=(2,0)$, $(0,2)$, or $(2,2)$, the cover $\phi_g$ is ramified at 6 points, so $X_g$ has genus $2$.
Also $X_g$ has an automorphism of order $6$. 

For $g=(2,4)$, the cover $\phi_g$ is unramified and so $X_{(2,4)}$ has genus $4$.
The group $\overline{G} = G/\langle (2,4) \rangle$ is isomorphic to $\ZZ/6 \ZZ \times \ZZ/2\ZZ$, generated by $\tau = (1,1)$ and
$\sigma = (1,2)$.  Note that $(a_0,a_1,a_\infty)$ reduces to $(2\tau + \sigma, 5\tau + \sigma, 5 \tau) \approx ((2,1), (5,1), (5,0))$.
In the Klein four action on $X_{(2,4)}$, each of the three involutions fixes two points, so the quotient by an involution gives a curve of 
genus $2$ that has an automorphism of order $6$.

By \cite[Section 8, page 645]{Igusa}, a curve of genus $2$ with an automorphism of order $6$ is isomorphic to $y^2 = x(x-1)(x+1)(x-2)(x-(1/2))$.  By \cite[Proposition~1.11]{IKO}, this curve is supersingular if and only if $p \equiv 5 \bmod 6$.
Applying Theorem~\ref{TKR} twice 
%(to the quotient of $X_{2,4}$ by the Klein four action and to the quotient of $X$ by $H$)  
shows that $X$ is supersingular if and only if $p \equiv 5 \bmod 6$.  

If $z=[9,9,9,3]$, then $G = \ZZ/9 \ZZ \times \ZZ/3 \ZZ$ and $a=((1,0), (1,1), (7,2))$.
Let $H=(\ZZ/3 \ZZ)^2$.
The quotients $X_{(3,0)}$ and $X/H$ are both elliptic curves, having an automorphism of order $3$.
Each of $X_{(3,1)}$, $X_{(3,2)}$ and $X_{(0,1)}$ has genus $4$, and is a $\ZZ/9\ZZ$-cover
of $\mathfont{P}^1$ branched at $B$ with inertia type $\bar{a} \approx (1,1,7)$.  Then use Theorem~\ref{TKR}.

If $z=[12,12,3,3]$, then $G = \ZZ/12 \ZZ \times \ZZ/3 \ZZ$ and 
$a=((1,0), (11,2), (0,1))$.
Let $H=(\ZZ/3 \ZZ)^2$.
The quotients $X_{(0,1)}$ and $X_H$ both have genus $0$.
By Theorem~\ref{TKR}, 
$\mathfont{Jac}(X) \sim \mathfont{Jac}(X_{(4,0)}) \times \mathfont{Jac}(X_{(4,1)}) \times \mathfont{Jac}(X_{(4,2)})$.
Then $X_{(4,0)}$ is a $\ZZ/4 \ZZ \times \ZZ/3\ZZ$-cover of $\mathfont{P}^1$ branched at $B$ 
with $a=((1,0), (3,2), (0,1))$, which is equivalent to a $\ZZ/12\ZZ$-cover with inertia type $(1,3,8)$.
Because $(\ZZ/12\ZZ \times \ZZ/3\ZZ)/\langle (4,1) \rangle$ is generated by 
$\tau = (1,2)$,  $X_{(4,1)}$ is a $\ZZ/12\ZZ$-cover with inertia type $(5 \tau, 3\tau, 4\tau) \approx (1,3,8)$.
Similarly, $(\ZZ/12\ZZ \times \ZZ/3\ZZ)/\langle (4,2) \rangle$ is generated by 
$\tau = (1,1)$ and $X_{(4,2)}$ is a $\ZZ/12\ZZ$-cover with inertia type $(5 \tau, 11\tau, 8\tau) \approx (1,4,7)$.

The case $[22,22,2,2]$ follows from Lemma~\ref{Lreduce2}.
\end{proof}

\section{Conclusions}

\begin{thm} \label{Tmainthm}
The primes $p$ for which there exists a supersingular curve $X$ of genus $5 \leq g \leq 10$
having an abelian cover of $\mathfont{P}^1$ branched at $B=\{0,1,\infty\}$ are:
\begin{enumerate}
\item for $g=5$, any $p  \equiv  7,11,19,23,31,41,47,59,71,79,83,89,91,101,103, 107, 119 \bmod 120$
and/or $p \equiv QNR \bmod 11$.
%$p  \equiv  7,11,19,23,31,41,47,59,71,79,83,89,91,101,103, 107, 119 \bmod 120$. 
Thus $\delta_5=25/32$.

\item for $g=6$, any $p$ satisfying at least one of the conditions:
$p \equiv 2 \bmod 3$; $p \equiv 2,3,4 \bmod 5$; $p \equiv 3,5,6 \bmod 7$; $p \not \equiv 1,3,9 \bmod 13$;
$p \equiv 15 \bmod 16$; and $p \equiv 13 \bmod 24$.  
Thus $\delta_6 = 507/512$.

\item for $g=7$, any $p$ such that $p \equiv 2 \bmod 3$ and/or $p \equiv 3 \bmod 4$.  Thus $\delta_7 = 3/4$.

\item for $g=8$, any $p$ satisfying at least one of the conditions:
$p \equiv 2 \bmod 3$; $p \equiv 2,3,4 \bmod 5$; $p \not \equiv 1 \bmod 17$; and $p \not \equiv 1,15 \bmod 32$.
Thus $\delta_8 = 1023/1024$.

\item for $g=9$, any $p$ satisfying at least one of the conditions:
$p \equiv 2 \bmod 3$; $p \equiv 3 \bmod 4$; $p \equiv 3,5,6 \bmod 7$; and $p \equiv QNR \bmod 19$.
Thus $\delta_9 = 15/16$.

\item for $g=10$, any $p$ satisfying at least one of the conditions:
$p \equiv 2 \bmod 3$; $p \equiv QNR \bmod 11$; $p \not \equiv 1, 11, 31 \bmod 40$; and $p \equiv 3, 19, 27 \pmod{28}$.
Thus $\delta_{10}=31/32$.

\end{enumerate}
\end{thm}

\begin{proof}
For $g=5$, this follows from Propositions~\ref{Pekedahl5cyclic}, \ref{Pekedahl5notcyclica}, 
and \ref{Pekedahl5notcyclicb}.
For $6 \leq g \leq 10$, this follows from the tables in Sections~\ref{Sdatacyclic} and \ref{Sdatanoncyclic}.
\end{proof}

\bibliographystyle{amsalpha}
\bibliography{supersingular5}

\end{document}